\documentclass[a4paper]{article}
\usepackage{hyperref}
\usepackage{graphicx}
\usepackage{mathrsfs}
\usepackage{enumerate}
\usepackage{amsxtra,amssymb,latexsym, amscd,amsthm,amsmath}
\usepackage{indentfirst}
\usepackage{color}
\usepackage[utf8]{inputenc}
\usepackage[mathscr]{eucal}
\usepackage{amsfonts}
\usepackage{graphics}
\usepackage{multirow}
\usepackage{array}
\usepackage{subfigure}
\usepackage{cite}
\usepackage{wrapfig}
\usepackage{tikz}
\usepackage{diagbox}

\usepackage{pgfplots}
\pgfplotsset{compat=1.15}
\usepackage{mathrsfs}
\usetikzlibrary{arrows}

\newcommand{\footremember}[2]{%
    \footnote{#2}
    \newcounter{#1}
    \setcounter{#1}{\value{footnote}}%
}
 
\def\R{{\mathbb R}}
\def\P{{\mathbb P}}

\def\N{{\mathbb N}}

\DeclareMathOperator{\rank}{rank}

\DeclareMathOperator{\bit}{bit}

\DeclareMathOperator{\sing}{sing}
\DeclareMathOperator{\reg}{reg}
\DeclareMathOperator{\KKT}{KKT}

\newtheorem{theorem}{\bf Theorem}
\newtheorem{lemma}{\bf Lemma}
\newtheorem{example}{\bf Example}
\newtheorem{remark}{\bf Remark}

\providecommand{\keywords}[1]
{
  \textbf{\textbf{Keywords:}} #1
}
\begin{document}
\definecolor{qqzzff}{rgb}{0,0.6,1}
\definecolor{ududff}{rgb}{0.30196078431372547,0.30196078431372547,1}
\definecolor{xdxdff}{rgb}{0.49019607843137253,0.49019607843137253,1}
\definecolor{ffzzqq}{rgb}{1,0.6,0}
\definecolor{qqzzqq}{rgb}{0,0.6,0}
\definecolor{ffqqqq}{rgb}{1,0,0}
\definecolor{uuuuuu}{rgb}{0.26666666666666666,0.26666666666666666,0.26666666666666666}
\newcommand{\vi}[1]{\textcolor{blue}{#1}}
\newif\ifcomment
\commentfalse
\commenttrue
\newcommand{\comment}[3]{%
\ifcomment%
	{\color{#1}\bfseries\sffamily#3%
	}%
	\marginpar{\textcolor{#1}{\hspace{3em}\bfseries\sffamily #2}}%
	\else%
	\fi%
}

\newcommand{\mapr}[1]{{{\color{blue}#1}}}
\newcommand{\revise}[1]{{{\color{blue}#1}}}
\newcommand{\victor}[1]{{{\color{red}#1}}}

\title{A Nichtnegativstellensatz on singular varieties under the denseness of regular loci}

\author{%
Ngoc Hoang Anh Mai\footremember{1}{University of Konstanz,
Universit\"atsstra{\ss}e 10, D-78464 Konstanz, Germany.}
  }

\maketitle

\begin{abstract}
Let $V$ be a real algebraic variety with singularities and  $f$ be a real polynomial non-negative on $V$.
Assume that the regular locus of $V$ is dense in $V$ by the usual topology.
Using Hironaka's resolution of singularities and Demmel--Nie--Powers' Nichtnegativstellensatz, we obtain a sum of squares-based representation that characterizes the non-negativity of $f$ on $V$.
This representation allows us to build up exact semidefinite relaxations for polynomial optimization problems whose optimal solutions are possibly singularities of the constraint sets.
\end{abstract}

\tableofcontents
%
\section{Introduction}

For almost 135 years since Hilbert's work \cite{hilbert1888darstellung}, representations of polynomials non-negative on semi-algebraic sets (Nichtnegativstellens\"atze for short) have been widely studied with influential applications to real-world problems.
One of these applications is to construct exact semidefinite programs for polynomial optimization problems (see Lasserre's hierarchy \cite{lasserre2001global}).
In other words, Nichtnegativstellens\"atze allow us to find the exact solutions for a non-convex problem through convex problems.
Although the first two representations by Hilbert--Artin \cite{hilbert1888darstellung,artin1927zerlegung} and Krivine--Stengle \cite{krivine1964anneaux,stengle1974nullstellensatz} did not directly produce such exact semidefinite programs, later ones have enabled us to achieve them under certain conditions.
In this work and our recent efforts \cite{mai2022exact,mai2022complexity2,mai2022sums}, these conditions have been gradually reduced to the most general case.
We list in Table \ref{tab:top.N-satz} different types of Nichtnegativstellens\"atze and their applicability to polynomial optimization.
\begin{table}
    \caption{Nichtnegativstellens\"atze and exact semidefinite programs (SDPs).}
    \label{tab:top.N-satz}
\begin{center}
\begin{tabular}{|m{0.7cm}|m{3.6cm}|m{5.2cm}|m{0.8cm}|}
\hline 
Year& Author(s) & Requirement 

(Method) & Exact SDPs\\ 
\hline
1888& Hilbert--Artin \cite{hilbert1888darstellung,artin1927zerlegung} & whole space,

non-prescribed denominators &no
\\
 \hline
1964 & Krivine--Stengle \cite{krivine1964anneaux,stengle1974nullstellensatz}& non-prescribed denominators & no
\\
  \hline
2000 & Scheiderer \cite{scheiderer2000sums,scheiderer2003sums,scheiderer2006sums} & low dimension,

compact basic semi-algebraic sets 

(local–global principle)
& yes
\\
\hline 
2006 & Marshall \cite{marshall2006representations,marshall2009representations} & Archimedean condition, 

boundary Hessian conditions

(local–global principle)
& yes 
 \\
\hline 
2006 & Nie--Demmel--Sturmfels \cite{nie2006minimizing} & whole space

(gradient ideals)
& yes 
 \\
 \hline
2007 & Demmel--Nie--Powers  \cite{demmel2007representations} & Karush--Kuhn--Tucker conditions
 & yes\\
 \hline
 2013 & Nie  \cite{nie2013polynomial} & finite real algebraic varieties
 & yes\\
 \hline
2022 & Mai  \cite{mai2022exact,mai2022complexity2}& finite images of singular loci

(Fritz John conditions)
& yes\\
 \hline
2022 & Mai (Theorem \ref{theo:rep.nonneg}) & denseness of regular loci

(resolution of singularities) & yes
 \\\hline
 2022 & Mai--Magron \cite{mai2022sums} & no 

(decomposition of singular loci) & yes\\
 \hline
 
\end{tabular}    
\end{center}
\end{table}
In this paper, we omit Positivstellens\"atze, representations of polynomials strictly positive on semi-algebraic sets, because they provide approximate semidefinite programs for polynomial optimization problems.

\paragraph{Previous works.} Hilbert and Artin \cite{hilbert1888darstellung,artin1927zerlegung} state and prove that every globally non-negative polynomial is a sum of squares of rational functions.
Krivine and Stengle \cite{krivine1964anneaux,stengle1974nullstellensatz} extend this result to the case of basic semi-algebraic sets, which are defined by polynomial inequalities.
It says that every polynomial non-negative on a basic semi-algebraic set can be expressed as the linear combination of the products of polynomials defining this set with weights being the sums of squares of rational functions.
Thus non-prescribed denominators exist in these two Nichtnegativstellens\"atze, which leads to difficulties in applying them to polynomial optimization.

Representation without denominators developed by Scheiderer \cite{scheiderer2000sums,scheiderer2003sums,scheiderer2006sums} allows him to replace sums of squares of rational functions in Krivine--Stengle's Nichtnegativstellensatz with sums of squares of polynomials.
However, this representation requires the boundedness of basic semi-algebraic sets in low-dimensional spaces.
Later Marshall \cite{marshall2006representations,marshall2009representations} improves this result in arbitrarily dimensional spaces using the Archimedean condition and the boundary Hessian conditions.
Marshall's Nichtnegativstellensatz states that every polynomial non-negative on a basic semi-algebraic set can be written as the linear combination of polynomials defining this set with weights being the sums of squares of polynomials.
Note that the number of sums of squares in this representation is not exponential as in the one by Scheiderer.
Moreover, the boundary Hessian conditions also partially reveal the use of optimality conditions in later representations.

Using the gradient of a non-negative polynomial $f$, Nie, Demmel, and Sturmfels \cite{nie2006minimizing} obtain a representation of $f$ with application in unconstrained polynomial optimization.
They claim that $f$ is identical to a sum of squares of polynomials on the set where the gradient of $f$ vanishes.
This result is extended to the case of $f$ non-negative on a basic semi-algebraic set $S$ by Demmel, Nie, and Powers \cite{demmel2007representations}.
It states that $f$ is concise to the linear combination of the products of polynomials defining $S$ with weights being sums of squares of polynomials on the set where the gradient of the Lagrangian for minimizing $f$ on $S$ vanishes.
Demmel--Nie--Powers' Nichtnegativstellensatz can be applied to polynomial optimization problems that the Karush--Kuhn--Tucker conditions hold at some global minimizer.
We prove in Lemma \ref{lem:KKT} that the Karush--Kuhn--Tucker conditions hold at regular points of the constraint sets.

Our previous works \cite{mai2022exact,mai2022complexity2} deal with the singular case of the constraint sets using the Fritz John conditions.
We obtain a similar Nichtnegativstellensatz to Demmel--Nie--Powers' by replacing the standard Lagrangian with a generalized form with a multiplier for the objective function.
We apply this representation to polynomial optimization in cases where the images of the singular loci of the constraint sets under the objective polynomials are finite.
(The motivation for the study of singularity theory in polynomial optimization is mentioned in \cite{mai2022symbolic}.)
So far, the answer to the existence of a general representation seems to be only a hair's breadth away.

\paragraph{Contribution.}
Let $V$ be a real algebraic variety (common real zeros of a system of polynomials) with singularities and $f$ be a real polynomial non-negative on $V$.
In order to provide a representation of $f$ as in Theorem \ref{theo:rep.nonneg}, this paper aims to convert $V$ to a regular variety $W$ through the resolution of singularities by Hironaka \cite{hironaka1964resolution,hironaka1964resolution2}.
More explicitly, there exists a proper birational
morphism $\varphi:W\to V$ such that $\varphi$ is an isomorphism over $V^{\reg}$, the regular locus of $V$.
Note that $\varphi$ is defined by a vector of polynomials.
We now assume that $V^{\reg}$ is dense in $V$ (by the usual topology).
Then the non-negativity of $f$ on $V$ is equivalent to the non-negativity of $f$ on $V^{\reg}$.
Since $\varphi$ is an isomorphism over $V^{\reg}$, $f$ is non-negative on $V^{\reg}$ if and only if $f\circ \varphi$ is non-negative on $W$.
Suppose that $f\circ \varphi$ attains its infimum value on $W$.
We then apply Demmel--Nie--Powers' Nichtnegativstellensatz for $f\circ \varphi$ non-negative on $W$.
Thus a new representation that characterizes the non-negativity of $f$ on $V$ is obtained. 
Consequently, in Theorem \ref{theo:pop}, we build up exact semidefinite relaxations based on this Nichtnegativstellensatz for a polynomial optimization problem.

For interested readers, the resolution of singularities is a powerful and frequently used tool in various applications of algebraic geometry.
It allows us to consider a singular variety as the image of some regular one, where problems are easily solved.
In \cite{watanabe2009algebraic}, Watanabe provides accurate estimation methods in singular statistical models using the resolution of singularities.
However, in the study of Nichtnegativstellensatz and its application to polynomial optimization, this tool has not been explored until now despite the increasing appearance of singular models (e.g., neural networks, HMMs, Bayesian networks, and stochastic context-free grammars).

\paragraph{Forthcoming work.}
Instead of using the resolution of singularities,
Mai and Magron \cite{mai2022sums} handle the general case of polynomial $f$ non-negative on a real algebraic variety $V$ by decomposing $V$  into finitely many subvarieties $V_1,\dots,V_r$. 
Here each $V_j$ is defined by higher-order derivatives of polynomials defining $V$.
Moreover, one of the following two cases occurs: (i) $V_j$ is the singular locus of some $V_i$, and (ii) $V_j$ has zero dimension.
If a polynomial $f$ is non-negative on $V$, then $f$ is non-negative on each $V_j$. 
In this case, the choice of using Demmel--Nie--Powers'  \cite{demmel2007representations} or Nie's Nichtnegativstellensatz \cite{nie2013polynomial} for $f$ over $V_j$ depends on the dimension of $V_j$.
Consequently, every polynomial optimization problem (particularly hyperbolic program) is equivalent to a semidefinite program.

\paragraph{Organization.}
We organize the paper as follows: 
Section \ref{sec:Preliminaries} presents some preliminaries from real algebraic geometry needed to prove our main results.
Section \ref{sec:represent.sing} is to state our Nichtnegativstellensatz on singular varieties. 
Section \ref{sec:POP.general} is to construct our exact semidefinite programs for polynomial optimization problems.

\section{Preliminaries}
\label{sec:Preliminaries}
\subsection{Real algebraic varieties}
Let $\R[x]$ denote the ring of polynomials with real coefficients in the vector of variables $x=(x_1,\dots,x_n)$.
Let $\R_r[x]$ denote the linear space of polynomials in $\R[x]$ of degree at most $r$.
Let $\|\cdot\|_2$ denote the $l_2$-norm of a real vector.
Then $\|x\|_2^2=x_1^2+\dots+x_n^2$ is a polynomial in $x$.

Given $h_1,\dots,h_l$ in $\R[x]$, we denote by $V(h)$ the (real) algebraic variety in $\R^n$ defined by the vector $h=(h_1,\dots,h_l)$, i.e.,
\begin{equation}
V(h):=\{x\in \R^n\,:\,h_j(x)=0\,,\,j=1,\dots,l\}\,.
\end{equation}
In this case, $h_1,\dots,h_l$ are called the polynomials defining $V(h)$. 

Given $h_1,\dots,h_l\in\R[x]$, let $I(h)[x]$ be the ideal generated by $h=(h_1,\dots,h_l)$, i.e.,
\begin{equation}
    I(h)[x]:= \sum_{j=1}^l h_j \R[x]\,.
\end{equation}
The real radical of an ideal $I(h)[x]$, denoted by $\sqrt[\R]{I(h)[x]}$, is defined as
\begin{equation}
{\sqrt[\R]{I(h)[x]}}:=\{f\in\R[x]\,:\,\exists m\in \N\,:\,-f^{2m}\in\Sigma^2[x]+I(h)[x]\}\,.
\end{equation}
Krivine--Stengle's Nichtnegativstellensatz \cite{krivine1964anneaux} imply that
\begin{equation}\label{eq:real.radi}
\sqrt[\R]{I(h)[x]}:=\{p\in\R[x]\,:\,p=0\text{ on }V(h)\}\,.
\end{equation}
We say that $I(h)[x]$ is real radical if $I(h)[x]=\sqrt[\R]{I(h)[x]}$.

Given an algebraic variety $V$ in $\R^n$, we denote by $I(V)$ the vanishing ideal of $V$, i.e.,
\begin{equation}
I(V):=\{p\in\R[x]\,:\,p(x)=0\,,\,\forall x\in V\}\,.
\end{equation}
Note that $I(V)$ is a real radical ideal.
To compute the generators of $I(V)$ with given polynomials defining $V$, we can use, e.g., Becker--Neuhaus' method in \cite{becker1993computation,neuhaus1998computation}. 
Thereby the degrees of the generators of $I(V)$ are bounded from above by $d^{2^{\mathcal O(n^2)}}$ if the degrees of polynomials defining $V$ are at most $d$.

Given an algebraic variety $V$ in $\R^n$, we say that $V$ is irreducible if there do not exist two proper subvarieties $V_1$, $V_2$ in $\R^n$ such that $V = V_1 \cup V_2$.
\subsection{Regular and singular loci}
Given $h=(h_1,\dots,h_l)$ with $h_j\in\R[x]$, we denote by $J(h)$ the Jacobian matrix associated with $h$, i.e.,
\begin{equation}
J(h)(x):=\left(\frac{\partial h_j}{\partial x_t}(x)\right)_{1\le t\le n,1\le j\le l}\,.
\end{equation}
We define the dimension of an algebraic variety $V$ in $\R^n$, denoted by $\dim(V)$, to be the highest dimension at points at which $V$ is a real submanifold.
For convenience we assume $\dim(\emptyset)=0$ in this paper.

Let $V$ be a algebraic variety in $\R^n$ of dimension $d$.
Let $h_1,\dots,h_l\in\R[x]$ be the generators of $I(V)$.
Set $h=(h_1,\dots,h_l)$.
We say that $a\in V$ is a regular point of $V$ if the Jacobian matrix $J(h)(a)$ has rank $n-d$.
Here the rank of a matrix $M$ with real coefficients is the largest integer $r$ such that all $(r+1)\times(r+1)$ minors of $M$ vanish.
The set $V^\text{reg}$ of all regular point of $V$ is called the regular locus of $V$.
Let $V^{\sing}=V\backslash V^\text{reg}$.
We say that $a\in V$ is a singular point of $V$ if $a\in V^{\sing}$, i.e., the Jacobian matrix $J(h)(a)$ has rank smaller than $n-d$.
The set $V^{\sing}$ is called the singular locus of $V$.
If $V^{\sing}\ne \emptyset$, we say that $V$ is singular. 
Otherwise, $V$ is called to be regular.
The tangent space of $V$ at $a\in V$, denoted by
$T_a(V)$, is the linear subspace of $\R^n$ given by
\begin{equation}\label{eq:tangent}
T_a(V):=\{u\in\R^n\,:\,J(h)(a)^\top u=0\}\,.
\end{equation}
It is not hard to prove that $a\in V$ is a regular point of $V$ iff $\dim(T_a(V))=d=\dim(V)$.

Given $h=(h_1,\dots,h_l)$ with $h_j\in\R[x]$, denote by $m_t(h)(x)$ the vector of $t\times t$ minors of the Jacobian matrix $J(h)(x)$.
Then $m_t(h)$ has length $\binom{n}{t}\times \binom{l}{t}$.
Each entry of $m_t(h)$ is in $\R[x]$ and has degree at most $t\times \max_j\deg(h_j)$.

The following lemma shows how to compute the singular locus of a real algebraic variety:
\begin{lemma}\label{lem:property.sing}
Let $V$ be an algebraic variety in $\R^n$ of dimension $d$.
Let $h_1,\dots,h_l\in\R[x]$ be the generators of $I(V)$.
Set $h=(h_1,\dots,h_l)$.
Let $V^{\sing}$ be the singular locus of $V$.
Then $V^{\sing}$ is an algebraic variety in $\R^n$ defined by $(h,m_{n-d}(h))$.
\end{lemma}
The proof of Lemma \ref{lem:property.sing} is similar to the one of \cite[Section 6.2]{smith2004invitation}.

As mentioned in \cite[Remark 3.3.15]{bochnak2013real}, the regular locus $V^{\reg}$ of a variety $V$ is not always dense in $V$.
The following example shows such a situation:
\begin{example}\label{exam:not.dense}
Consider $V$ as the Whitney umbrella, i.e., $V=V(h)$ with $h=x_1^2-x_2^2x_3$. 
By Lemma \ref{lem:property.sing}, the singular locus $V^{\sing}$ of $V$ is the $x_3$-axis, i.e., $V^{\sing}=V(x_1,x_2)$. 
Let us prove that the regular locus $V^{\reg}$ of $V$ is contained in the half space $x_3\ge 0$, i.e., $V^{\reg}\subset\{x\in\R^3\,:\,x_3\ge 0\}$.
Assume by contradiction that $x\in V^{\reg}$ and $x_3< 0$.
Since $h(x)=0$, we get $x_1=x_2=0$.
It implies that $x\in V^{\sing}$, which is impossible since $V^{\reg}=V\backslash V^{\sing}$.
Thus there is no point $x$ in $V^{\sing}$ with $x_3<0$ belonging to the closure of $V^{\reg}$ (by the usual topology).
\end{example}
The readers might wonder if the regular locus $V^{\reg}$ of an irreducible variety $V$ is always dense in $V$.
The answer is still no, as shown in the following example:
\begin{example}\label{exam:not.dense2}
Consider $V$ as the Cartan umbrella, i.e., $V=V(h)$ with $h=x_3(x_1^2+x_2^2)-x_1^3$.
As mentioned in \cite[Example 3.1.2 d)]{bochnak2013real}, $V$ is irreducible.
By Lemma \ref{lem:property.sing}, the singular locus $V^{\sing}$ of $V$ is the $x_3$-axis, i.e., $V^{\sing}=V(x_1,x_2)$.
Let $a\in V^{\sing}$ such that $a_3=1$.
Then $a_1=a_2=0$.
 Let us prove that there is no point in the intersection of $V^{\reg}$ with a sufficiently small neighborhood of $a$.
Let $\varepsilon\in (0,\frac{1}{2})$.
Assume by contradiction that $b$ is in the intersection of $V^{\reg}$ and the open ball centered at the origin of radius $\varepsilon$.
Then $\varepsilon^2>\|a-b\|_2^2=b_1^2+b_2^2+(b_3-1)^2$.
It implies that 
\begin{equation}\label{eq:ineq.esp}
b_3>1-\varepsilon\quad\text{and}\quad \varepsilon>b_1\,.
\end{equation}
Since $b\in V^{\reg}$, we get $b_1^2+b_2^2> 0$ and $b_3=\frac{b_1^3}{b_1^2+b_2^2}$.
From \eqref{eq:ineq.esp}, it follows that 
\begin{equation}\label{eq:ineq.esp.2}
\frac{b_1^3}{b_1^2+b_2^2}>1-\varepsilon\,.
\end{equation}
Since $1-\varepsilon>\frac{1}{2}$, we obtain $b_1>0$.
By \eqref{eq:ineq.esp.2}, we get  
\begin{equation}
b_1^3>(1-\varepsilon)(b_1^2+b_2^2)>(1-\varepsilon)b_1^2
\end{equation}
since $(1-\varepsilon)b_2^2\ge 0$.
It implies that $b_1>1-\varepsilon$.
By \eqref{eq:ineq.esp}, $\varepsilon>b_1>1-\varepsilon$ gives $\varepsilon>\frac{1}{2}$. 
It is impossible since $\varepsilon<\frac{1}{2}$.
\end{example}

However, there are several cases where the regular locus $V^{\reg}$ of $V$ is dense in $V$.
Let us consider the following three examples:
\begin{example}\label{exam:cuspidal.plane}
Let $V$ be the cuspidal plane curve, i.e., $V=V(p)$ with $p=x_1^3 -x_2^2$.
By Lemma \ref{lem:property.sing}, the singular locus $V^{\sing}$ of $V$ is singleton, i.e., $V^{\sing}=\{(0,0)\}$.
Moreover, the regular locus $V^{\reg}$ of $V$ is dense in $V$.
Indeed, it is sufficient to find a regular point of $V$ in any neighborhood of the origin.
Let $\varepsilon>0$.
Set $a=(\varepsilon^{2},\varepsilon^3)$.
Then we get $a_2\ne 0$, $p(a)=0$, and $\|a-0\|_2=\varepsilon^2\sqrt{1+\varepsilon^2}\to 0$  as $\varepsilon\to 0^+$.
It implies that $a$ is in the intersection of $V^{\reg}$ with an arbitrarily small neighborhood of the origin. 
\end{example}

\begin{example}\label{exam:lorentz.cone}
Let $V$ be the boundary of the Lorentz cone in $\R^3$, i.e., $V=V(p)$ with $p=x_1^2 +x_2^2 - x_3^2$.
By Lemma \ref{lem:property.sing}, the singular locus $V^{\sing}$ of $V$ is singleton, i.e., $V^{\sing}=\{(0,0,0)\}$.
Moreover, the regular locus $V^{\reg}$ of $V$ is dense in $V$.
Indeed, it is sufficient to find a regular point of $V$ in any neighborhood of the origin.
Let $\varepsilon>0$.
Set $a=\frac{1}{2}(\varepsilon,\varepsilon,\sqrt{2}\varepsilon)$.
Then we get $a_3\ne 0$, $p(a)=0$, and $\|a-0\|_2=\varepsilon\to 0$ as $\varepsilon\to 0^+$.
It implies that $a$ is in the intersection of $V^{\reg}$ with an arbitrarily small neighborhood of the origin.
\end{example}

\begin{example}\label{exam:infinite.sing}
Let $V=V(p)$ with $p=x_1^2-x_3^2(x_3+x_2^2)$ (modified from \cite[Figure 9]{hauser2003hironaka}).
By Lemma \ref{lem:property.sing}, the surface $V$ has its singular points along
the $x_2$-axis, i.e., $V^{\sing}=\{x\in\R^3\,:\,x_1=x_3=0\}$. 
Let us prove that the regular locus $V^{\reg}$ of $V$ is dense in $V$.
Let $x\in V^{\sing}$. 
Then $x_1=x_3=0$.
We will find a regular point $y$ in any neighborhood of $x$.
Let $\varepsilon>0$.
Set $y_2=x_2$, $y_3=\varepsilon$ and $y_1=\varepsilon\sqrt{\varepsilon+x_2^2}$.
Then we get $y=(y_1,y_2,y_3)\in V^{\reg}$ since $y_3\ne 0$.
Moreover, $\|x-y\|_2=\varepsilon\sqrt{1+\varepsilon+x_2^2}\to 0$ as $\varepsilon\to 0^+$.
It implies that $y$ is in the intersection of $V^{\reg}$ with an arbitrarily small neighborhood of $x$. 
\end{example}

\subsection{Resolution of singularities}

A morphism $\varphi:W\to V$, between two varieties $W\subset \R^t$ and $V\subset \R^n$,
is given by $\varphi(y)=(\varphi_1(y),\dots,\varphi_n(y))$ for some polynomial $\varphi_j\in\R[y]$.
Here $y=(y_1,\dots,y_t)$ is a vector of variables.
If $\varphi_j$s are rational functions in $\R(y)$, $\varphi$ is called a rational map.
In this case, $\varphi$ is well-defined on an open subset of $W$.
A birational map from variety $W\subset \R^t$ to varieties $V\subset \R^t$ is a rational map $\varphi:W\to V$ such that there is a rational map $V\to W$ inverse to $\varphi$. 
We say that $\varphi:W\to V$ is a birational morphism, if $\varphi$ is a morphism which is birational.
An isomorphism is an invertible morphism.
A function $F:X\to Y$ between topological spaces $X$ and $Y$ is called proper if for every compact subset $A\subset Y$, the inverse image $F^{-1}(A)$ is compact in $X$.

We state Hironaka's theorem \cite{hironaka1964resolution,hironaka1964resolution2} in the following lemma (see also \cite[Theorem 1.0.3]{bierstone2011effective}):
\begin{lemma}\label{lem:resol.sing}
Let $V$ be a real algebraic variety with regular locus $V^{\reg}$.
Then there exists a regular variety $W$ together with a proper birational
morphism $\varphi:W\to V$ such that the restriction $U\to V^{\reg}\,,\,y\mapsto \varphi(y)$, is an isomorphism for some $U\subset W$.
\end{lemma}

Roughly speaking about the meaning of Lemma \ref{lem:resol.sing}, any real algebraic variety can be seen as an image of a regular real algebraic variety.
We illustrate the result of Lemma \ref{lem:resol.sing} in the following two examples:

\begin{example}\label{exam:resol.sing.1}
Let $p$, $V$, $V^{\sing}$, and $V^{\reg}$ be as in Example \ref{exam:cuspidal.plane} (see \cite[Section 2.1.4]{abramovich2018resolution}).
Let  $W=V(h)$ with $h=y_1-y_2^2\in\R[y]$.
Here $y=(y_1,y_2)$.
Then the variety $W$ is regular.
Let $\varphi:W\to V$ be the morphism defined by $\varphi(y_1,y_2)=(y_1,y_1y_2)$.
Let us prove that $\varphi$ is well-defined, i.e., $\varphi(W)\subset V$.
Let $x\in \varphi(W)$. Then there is $y\in W$ such that $x=\varphi(y)=(y_1,y_1y_2)$.
It implies that $p(x)=x_1^3-x_2^2=y_1^2(y_1-y_2^2)=y_1^2h(y)=0$, which yields $x\in V$.
Moreover, $\varphi$ is birational since $\varphi^{-1}(x_1,x_2)=(x_1,\frac{x_2}{x_1})$.
Note that $\varphi^{-1}$ is well-defined on $V^{\reg}$.
It is not hard to prove that the restriction $U=W\backslash\{(0,0)\}\to V^{\reg}\,,\,y\mapsto\varphi(y)$, is an isomorphism.
\end{example}
\begin{example}\label{exam:resol.sing}
Let $p$, $V$, $V^{\sing}$, and $V^{\reg}$ be as in Example \ref{exam:lorentz.cone} (see \cite{smith2016jyvaskyla}).
Let $W$ be the boundary of a cylinder in $\R^3$, i.e.,  $W=V(h)$ with $h=y_1^2+y_2^2-1\in\R[y]$.
Here $y=(y_1,y_2,y_3)$.
Then the variety $W$ is regular.
Let $\varphi:W\to V$ be the morphism defined by $\varphi(y_1,y_2,y_3)=(y_1y_3,y_2y_3,y_3)$.
Let us prove that $\varphi$ is well-defined, i.e., $\varphi(W)\subset V$.
Let $x\in \varphi(W)$. Then there is $y\in W$ such that $x=\varphi(y)=(y_1y_3,y_2y_3,y_3)$.
It implies that $y_3=x_3$.
If $x_3=0$, then $y_3=0$ yields $x=0\in V$.
Assume that $x_3\ne 0$.
Then $y_1=\frac{x_1}{x_3}$ and   $y_2=\frac{x_2}{x_3}$.
Since $y\in W$, we get $h(y)=y_1^2+y_2^2-1=0$, which gives $\frac{x_1^2}{x_3^2}+\frac{x_2^2}{x_3^2}-1=0$.
Thus $p(x)=0$ yields $x\in V$.
Moreover, $\varphi$ is birational since $\varphi^{-1}(x_1,x_2,x_3)=(\frac{x_1}{x_3},\frac{x_2}{x_3},x_3)$.
Note that $\varphi^{-1}$ is well-defined on $V^{\reg}$.
Set $C=W\cap \{y\in\R^3\,:\,y_3=0\}$.
It is not hard to prove that the restriction $U=W\backslash C\to V^{\reg}\,,\,y\mapsto\varphi(y)$, is an isomorphism.
Geometrically, the singular locus $V^{\sing}$ of $V$ is replaced with the circle $C$.
\end{example}
We refer the reader to \cite{ellwood2014resolution} for more interesting examples of the resolution of singularities.

\begin{remark}\label{rem:alg.resol}
In \cite{bierstone2011effective}, Bierstone, Grigoriev, Milman, and W{\l}odarczyk provide an algorithm for finding regular variety $W$ and morphism $\varphi$ in Lemma \ref{lem:resol.sing} with given variety $V$. 
They also analyze the complexity of their algorithm.
(We do not cover the algorithm of Bierstone, Grigoriev, Milman, and W{\l}odarczyk in this paper because it is  long and complex.)
The main idea of Hironaka is to blow up the singular locus of a given variety into a projective space. 
Repeating such a step several times allows all singular points of the input variety to be eliminated.

Let us recall the original concept of blow-up in a real projective space.
We first define an equivalence relation $\sim$ to the set $\R^{n}\backslash\{0\}$ by
\begin{equation}
x\sim y \Leftrightarrow \exists\ t \ne 0\,:\, x=ty\,.
\end{equation}
The quotient set $\P^{n-1}=\R^{n}/ \mathord{\sim}$ is said to be a $(n-1)$-dimensional real projective space.
The equivalence class that contains $x\in \R^{n} \backslash\{0\}$ is denoted by $(x_1:\dots:x_n)$.
Let $V$ and $W$ ($V\subset W$) be real algebraic sets in $\R^n$ with $h_1,\dots,h_l$ being the generators of $I(V)$.
The blow-up of $W$ with center $V$, denoted by $B_V(W )$, is defined by the closure
\begin{equation}
B_V(W)=\overline{\{(x,(h_1(x):\dots:h_l(x)))\,:\, x \in W\backslash V \}}^Z
\end{equation}
in the Zariski topology of $\R^n\times\P^l$ (see \cite[Definition 3.13]{watanabe2009algebraic}).
\end{remark}
\begin{remark}\label{rem:de.Jong}
In \cite{de1996smoothness}, de Jong proves that
for any variety $V$, there exists a regular variety $W$ with a proper morphism $\varphi: W\to V$, which is dominant, i.e., $\varphi(W)$ is dense in $V$ (by the usual topology).
As mentioned by Oort \cite{oort1999alterating}, no algorithm produces $W$ and $\varphi$ through de Jong's approach.
\end{remark}
\subsection{The Karush--Kuhn--Tucker conditions}
Given $h_0,h_1,\dots,h_l$ in $\R[x]$, consider the following polynomial optimization problem:
\begin{equation}\label{eq:pop}
    h^\star:=\inf\limits_{x\in V(h)} h_0(x)\,,
\end{equation}
where $V(h)$ is the algebraic variety in $\R^n$ defined by $h=(h_1,\dots,h_l)$.
\begin{remark}\label{rem:convert}
The more general form 
\begin{equation}
\begin{array}{rl}
\inf\limits_{y\in\R^r}&f(y)\\
\text{s.t.}&g_j(y)\ge 0\,,\,j=1,\dots,l\,,
\end{array}
\end{equation}
can be written as an instance of \eqref{eq:pop} by setting $x=(y,z)$, $h_0(x)=f(y)$, $h_j(x)=g_j(y)-z_j^2$.
\end{remark}
Given $p\in\R[x]$, we denote by $\nabla p$ the gradient of $p$, i.e., $\nabla p=(\frac{\partial p}{\partial x_1},\dots,\frac{\partial p}{\partial x_n})$.
We recall the Karush--Kuhn--Tucker conditions in the following lemma:
\begin{lemma}\label{lem:KKT}
Let $h_0$ in $\R[x]$. 
Let $V$ be an algebraic variety in $\R^n$ of dimension $d>0$. 
Let $h_1,\dots,h_l\in\R[x]$ be the generators of $I(V)$.
Let $x^\star$ be a local minimizer for problem \eqref{eq:pop} with $h=(h_1,\dots,h_l)$.
Assume that $J(h)(x^\star)$ has rank $n-d$.
Then the Karush--Kuhn--Tucker conditions hold for problem \eqref{eq:pop} at $x^\star$, i.e.,
\begin{equation}\label{eq:KKT}
    \begin{cases}
    		\exists (\lambda_1^\star,\dots,\lambda_l^\star)\in \R^{l}\,:\\
    		h_j(x^\star)=0\,,\,j=1,\dots,l\,,\\
          \nabla h_0(x^\star)=\sum_{j=1}^l \lambda_j^\star \nabla h_j(x^\star)\,.
    \end{cases}
\end{equation}
\end{lemma}
\begin{proof}
By assumption, $x^\star$ is a regular point of the manifold $V\subset \R^n$.
Then there exists a diffeomorphism $\Phi:U\to V$ for some open set $U\subset \R^d$ such that $x^\star=\Phi(t^\star)$ for some $t^\star\in U$.
The differential of $\Phi$ at $t\in U$ is defined by the linear mapping $D \Phi_t: \R^d\to T_{\Phi(t)}V$, $u\mapsto J(\Phi)(t)^\top u$, where $T_{a}V$ is the tangent space of $V$ at $a\in V$ (defined as in \eqref{eq:tangent}).
Since $x^\star$ is a regular points of $V$,  $D \Phi_{t^\star}$ is bijective.
From this, we get $\rank (J(\Phi)(t^\star)) =d$, which gives the null space of $J(\Phi)(t^\star)$, denoted by $	J(\Phi)(t^\star)^\perp$, has dimension $n-d$ thanks to the rank--nullity theorem.
By assumption, $t^\star$ is a local minimizer of $h_0\circ \Phi$ on $U$.
It implies that $0=\nabla (h_0\circ \Phi)(t^\star)=J(\Phi)(t^\star)\times \nabla h_0(x^\star)$, which gives $\nabla h_0(x^\star)$ is in $J(\Phi)(t^\star)^\perp$.
In addition, for $j=1,\dots,l$, $(h_j\circ \Phi)(t)=0$ for all $t\in U$. 
Take the gradient in $t$, we obtain $J(\Phi)(t)\times \nabla h_j(\Phi(t))=0$, for all $t\in U$, for $j=1,\dots,l$.
It implies that $\nabla h_j(\Phi(t))$, $j=1,\dots,l$, are in the null space of $J(\Phi)(t)$, for all $t\in U$.
By assumption, the linear span of $\nabla h_j(x^\star)$, $j=1,\dots,l$, has dimension $n-d$.
Since $J(\Phi)(t^\star)^\perp$ has dimension $n-d$,  $J(\Phi)(t^\star)^\perp$ is the linear span of $\nabla h_j(x^\star)$, $j=1,\dots,l$.
Hence the result follows since $\nabla h_0(x^\star)$ is in $J(\Phi)(t^\star)^\perp$.
\end{proof}
\begin{remark}
As shown in Freund's lecture note \cite[Theorem 11]{freund2004optimality}, the Karush--Kuhn--Tucker conditions \eqref{eq:KKT} hold for problem \eqref{eq:pop} at $x^\star$ when the linear independence constraint qualification is satisfied, i.e., the gradients $\nabla h_j(x^\star)$, $j=1,\dots,l$, are linearly independent in $\R^n$, which is equivalent to that $J(h)(x^\star)$ has rank $l$.
For comparison purposes, we make a weaker assumption in Lemma \ref{lem:KKT} that $J(h)(x^\star)$ has rank $n-\dim(V)$. 
Similarly to \cite[Exercise 17 b, page 495]{cox2013ideals}, we obtain $l\ge n-\dim(V)$ in general.
Note that the twisted cubic $V = \{(t,t^2,t^3)\in \R^3\,:\,t\in \R\}$ is a one-dimensional variety defined by $x_2-x_1^2$ and $x_3-x_1^3$, but the ideal $I(V)$ is generated by the vector of three polynomials $h=(x_1x_3-x_2^2,x_2-x^2_1,x_3-x_1 x_2)$. Thus it holds that $l=3>2=n-\dim(V)$ in this example.
\end{remark}

To prove that the largest rank assumption of $J(h)(x^\star)$ in Lemma \ref{lem:KKT} cannot be removed, consider the following example:
\begin{example}\label{exam:KKT.not.hold}
Let $n=2$, $h_0=x_1$, and $h=x_1^3-x_2^2$.
Then $x^\star=(0,0)$ is the unique global minimizer for problem \eqref{eq:pop}.
Moreover, the Karush--Kuhn--Tucker conditions do not hold for problem \eqref{eq:pop} at $x^\star$.
Indeed, for any $\lambda\in\R$, we get
\begin{equation}
\nabla h_0(x^\star)-\lambda\nabla h(x^\star)=\begin{bmatrix}
1\\
0
\end{bmatrix}-\lambda\begin{bmatrix}
3x_1^{*2}\\
-2x_2^\star
\end{bmatrix}=\begin{bmatrix}
1\\
0
\end{bmatrix}\ne 0\,.
\end{equation}
Note that $V(h)$ has dimension $d=1$, and $J(h)=\nabla h=\begin{bmatrix}
3x_1^{2}\\
-2x_2
\end{bmatrix}$.
It is not hard to check that $J(h)(x^\star)$ has rank  $0< 1=n-d$.
\end{example}

Given $\bar h:=(h_0,h)$ with $h:=(h_1,\dots,h_l)$ and  $h_j\in\R[x]$, we denote by $\bar h_{\KKT}$ the vector of polynomials in $\R[x, \lambda]$ associated with the Karush--Kuhn--Tucker conditions defined by
\begin{equation}
\bar h_{\KKT}:=(h,\nabla h_0-\sum_{j=1}^l \lambda_j \nabla h_j)\,,
\end{equation}
where $\lambda=(\lambda_1,\dots,\lambda_l)$.
The condition \eqref{eq:KKT} can be written as $(x^\star,\lambda^\star)\in V(\bar h_{\KKT})$ for some $\lambda^\star\in\R^l$.
\subsection{Semi-algebraic set}

Given $g=(g_1,\dots,g_m)$ with $g_j\in\R[x]$, we denote by $S(g)$ the basic semi-algebraic set associated with $g$, i.e.,
\begin{equation}
    S(g):=\{x\in\R^n\,:\,g_j(x)\ge 0\,,\,j=1,\dots,m\}\,.
\end{equation}

A semi-algebraic subset of $\R^n$ is a subset of the following form
\begin{equation}\label{eq:def.semi.set}
\bigcup_{i=1}^t\bigcap_{j=1}^{r_i}\{x\in\R^n\,:\,f_{ij}(x)*_{ij}0\}\,,
\end{equation}
where $f_{ij}\in\R[x]$ and $*_{ij}\in\{>,=\}$.
Note that \eqref{eq:def.semi.set} is the union of finitely many basic semi-algebraic sets.

Given two semi-algebraic sets $A\subset \R^n$ and $B\subset \R^m$, we say that a mapping $f : A \to B$ is semi-algebraic if its graph $\{(x,f(x))\,:\,x\in A\}$ is a semi-algebraic set in $\R^{n+m}$.
A semi-algebraic subset $A\subset \R^n$ is said to be semi-algebraically path connected if for every $x,y$ in $A$, there exists a continuous semi-algebraic mapping $\phi:[0,1] \to A$ such that $\phi(0) = x$ and $\phi(1) = y$.

The following lemma can be found in \cite[Proposition 1.6.2 (ii)]{pham2016genericity}:
\begin{lemma}\label{lem:composit.semi-al}
Compositions of semi-algebraic maps are semi-algebraic.
\end{lemma}

The following lemma is given in \cite[Theorem 1.8.1]{pham2016genericity}:
\begin{lemma}\label{lem:semial.func.anal}
Let $f:(a, b)\to \R$ be a semi-algebraic
function. Then there are $a = a_0 < a_1 < \dots < a_s < a_{s+1} = b$ such that, for each $i = 0,\dots, s$, the restriction $f|_{(a_i,a_{i+1})}$ is analytic.
\end{lemma}
The following lemma follows from the mean value theorem:
\begin{lemma}\label{lem:mean.val}
Let $f:[0,1]\to \R$ be a continuous piecewise-differentiable function, i.e., there exist $0=a_1<\dots<a_r=1$ such that $f$ is continuous, and $f$ is differentiable on each open interval $(a_i,a_{i+1})$.
Assume that $f$ has zero subgradient.
Then $f(0)=f(1)$.
\end{lemma}
\begin{proof}
By using the mean value theorem on each open interval $(a_i,a_{i+1})$, we get $f(a_i)=f(a_{i+1})$.
Hence $f(0)=f(a_1)=\dots=f(a_{r})=f(1)$ yields the result.
\end{proof}

Given $n,d,s\in\N$, we define
\begin{equation}\label{eq:bound.connected}
c(n,d,s):=d(2d-1)^{n+s-1}\,.
\end{equation}

We recall in the following lemma the upper bound on the number of connected components of a basic semi-algebraic set is stated by Coste in \cite[Proposition 4.13]{coste2000introduction}:
\begin{lemma}\label{lem:num.connected}
Let $g_1,\dots,g_m,h_1,\dots,h_l\in\R_d[x]$ with $d\ge 2$. 
The number of (semi-algebraically path) connected components of $S(g)\cap V(h)$ is not greater than $c(n,d,m+l)$.
\end{lemma}
\subsection{Sums of squares}
Denote by $\Sigma^2[x]$ (resp. $\Sigma^2_r[x]$) the cone of sums of squares of polynomials in $\R[x]$ (resp. $\R_r[x]$).
Given $g_1,\dots,g_m\in\R[x]$, let $P_r(g)[x]$ be the truncated preordering of order $r\in\N$ associated with $g=(g_1,\dots,g_m)$, i.e., 
\begin{equation}
    P_r(g)[x]:=\{\sum_{\alpha\in\{0,1\}^m} \sigma_\alpha g^\alpha\,:\,\sigma_\alpha\in\Sigma^2[x]\,,\,\deg(\sigma_\alpha g^\alpha)\le 2r\}\,,
\end{equation}
where $\alpha=(\alpha_1,\dots,\alpha_m)$ and $g^\alpha:=g_1^{\alpha_1}\dots g_m^{\alpha_m}$.
If $m=0$, it holds that $P_r(g)[x]=\Sigma^2_r[x]$.

Given $h_1,\dots,h_l\in\R[x]$, let $I_r(h)[x]$ be the truncated ideal of order $r$ defined by $h$, i.e.,
\begin{equation}
    I_r(h)[x]:= \{\sum_{j=1}^l h_j \psi_j\,:\,\psi_j\in\R[x]\,,\,\deg(h_j \psi_j)\le 2r\}\,.
\end{equation}

We denote by $\bit(d)$ the number of bits of $d\in\N$, i.e.,
\begin{equation}
\bit(d):=
 \begin{cases}
          1 & \text{if } d = 0\,,\\
	    k & \text{if } d \ne 0 \text{ and } 2^{k-1}\le d < 2^k.
         \end{cases}
\end{equation}
Given $n,d,s\in\N$, we define 
\begin{equation}\label{eq:bound.Krivine}
b(n,d,s):=2^{
2^{\left(2^{\max\{2,d\}^{4^{n}}}+s^{2^{n}}\max\{2, d\}^{16^{n}\bit(d)} \right)}}\,.
\end{equation}

We recall the degree bounds for Krivine--Stengle's Nichtnegativstellens\"atze by Lombardi, Perrucci, and Roy \cite{lombardi2020elementary} in the following two lemmas:
\begin{lemma}\label{lem:pos}
Let $g_1,\dots,g_m,h_1,\dots,h_l$ in $\R_d[x]$. 
Assume that $S(g)\cap V(h)=\emptyset$ with $g:=(g_1,\dots,g_m)$ and $h:=(h_1,\dots,h_l)$. 
Set $r=\frac{1}{2}\times b(n,d,m+l+1)$.
Then it holds that $-1 \in P_r(g)[x]+I_r(h)[x]$.
\end{lemma}
\begin{lemma}\label{lem:pos2}
Let $p,g_1,\dots,g_m,h_1,\dots,h_l$ in $\R_d[x]$. 
Assume that $p$ vanishes on $S(g)\cap V(h)$ with $g:=(g_1,\dots,g_m)$ and $h:=(h_1,\dots,h_l)$. 
Set $r:=\frac{1}{2}\times b(n,d,m+l+1)$ and $s:=2\lfloor r/d\rfloor$.
Then it holds that $-p^s \in P_r(g)[x]+ I_r(h)[x]$.
\end{lemma}
\subsection{Nichtnegativstellensatz on regular varieties}
\label{sec:proof.rep}
Denote by $|\cdot|$ the cardinality of a set and by $\delta_{ij}$ the Kronecker delta function at $(i,j)\in\N^2$.

We state in the following lemma a sums of squares-based representation with degree bound for a polynomial which has finitely many non-negative values on a real algebraic variety:
\begin{lemma}\label{lem:quadra}
Let $h_0,h_1,\dots,h_l$ in $\R_{d}[x]$. 
Assume that $h_0$ is non-negative on $V(h)$  and $h_0(V(h))$ is finite with $h=(h_1,\dots,h_l)$. 
Set $r:=|h_0(V(h))|$ and $u:=\frac{1}{2}\times b(n,d,l+1)$.
Then there exists $\sigma\in \Sigma^2_w[x]$ with $w=\max\{d(r-1),u+d\}$ such that $h_0 - \sigma$ vanishes on $V(h)$.
\end{lemma}
\begin{proof}
Consider the following two cases:
\begin{itemize}
\item Case 1: $r=0$. It is obvious that  $V(h)=\emptyset$.
Lemma \ref{lem:pos} says that $-1 \in \Sigma^2_u[x]+I_u(h)[x]$.
It implies that there exists $q \in P_u(g)[x]$ such that
$-1 = q$ on $V(h)$. 
We write $h_0 = s_1 - s_2$, where $s_1 =(h_0+\frac{1}{2})^2$ and $s_2 = h_0^2+\frac{1}{4}$ are in $\Sigma^2[x]$.
From this we get $h_0 = s_1 + q s_2$ on $V(h)$.
Letting $\sigma = s_1 +qs_2$ gives $\sigma\in \Sigma^2_{u+d}[x]\subset \Sigma^2_{w}[x]$ since $w\ge u+d$.
Thus $h_0-\sigma$ vanishes on $V(h)$.

\item Case 2: $r>0$. By assumption, we can assume that
$h_0(V(h)) = \{t_1 ,\dots, t_r \} \subset [0,\infty)$,
where $t_i\ne t_j$ if $i\ne j$.
For $j=1,\dots,r$, let
$W_j:=V(h,h_0-t_j)$.
Then $W_j$ is a real variety defined by $l+1$ polynomials in $\R_{d}[x]$.
It is clear that $h_0(W_j)=\{t_j\}$.
Define the following polynomials:
\begin{equation}\label{eq:lagrange.pol}
p_j(x):=\prod_{i\ne j}\frac{h_0(x)-t_i}{t_j-t_i}\,,\,j=1,\dots,r\,.
\end{equation}
It is easy to check that $p_j(W_i)=\{\delta_{ji}\}$ and $\deg(p_j)\le d(r-1)$.
Note that $h_0=t_i\ge 0$ on $W_i$, for $i=1,\dots,r$.  
Now letting $\sigma =\sum_{i=1}^r t_i p_i^2$,
we obtain $\sigma\in \Sigma^2_w[x]$ since $\deg(t_i p_i^2)\le 2\deg(p_i)\le 2d(r-1)\le 2w$.
Hence $h_0 - \sigma$ vanishes on $V(h)=W_1\cup\dots\cup W_r$, yielding the result.
\end{itemize}
\end{proof}

The following lemma is similar to \cite[Lemma 3.3]{demmel2007representations} but is proved by using the tools from real algebraic geometry (instead of the ones from complex algebraic geometry):
\begin{lemma}\label{lem:constant.KKT}
Let $h_1,\dots,h_l$ in $\R[x]$.
Let $h_0$ be a polynomial in $\R[x]$.
Let $W$ be a semi-algebraically path connected component of $V(\bar h_{\KKT})$, where $\bar h:=(h_0,\dots,h_l)$. Then $h_0$ is constant on $W$.
\end{lemma}
\begin{proof} 
Recall $\lambda:=(\lambda_1,\dots,\lambda_m)$.
Choose two arbitrary points $(x^{(0)},\lambda^{(0)})$, $(x^{(1)},\lambda^{(1)})$ in $W$. 
We claim that $h_0(x^{(0)}) = h_0(x^{(1)})$.
By assumption, there exists a continuous semi-algebraic mapping $\phi:[0,1]\to W$ defined by $\phi(\tau) = (x(\tau),\lambda(\tau))$ such that $\phi(0) = (x^{(0)},\lambda^{(0)})$ and $\phi(1) = (x^{(1)},\lambda^{(1)})$. 
We claim that $\tau\mapsto h_0(x(\tau))$ is constant on $[0,1]$.
The Lagrangian function
\begin{equation}\label{eq:Lagran.KKT}
    L(x,\lambda) := h_0(x)-\sum_{j=1}^m \lambda_j h_j (x)\,
\end{equation}
is equal to $h_0(x)$ on $V(\bar h_{\KKT})$, which contains $\phi([0,1])$. 
By Lemma \ref{lem:composit.semi-al}, the function $L\circ \phi$ is semi-algebraic.
Moreover, the function $L\circ \phi$ is continuous since $L$ and $\phi$ are continuous.
It implies that $L\circ \phi$ is a continuous piecewise-differentiable function thanks to Lemma \ref{lem:semial.func.anal}.
Note that the function
$L\circ \phi$
has zero subgradient on $[0,1]$.
From Lemma \ref{lem:mean.val}, it follows that $h_0(x (0) )=(L\circ \phi)(0)= (L\circ \phi)(1)= h_0(x (1) )$.
We now obtain $h_0(x^{(0)})$ = $h_0(x^{(1)})$, and hence $h_0$ is constant on $W$.
\end{proof}

Based on the Karush--Kuhn--Tucker conditions, we state in the following theorem the sums of squares-based representation with degree bound for a polynomial non-negative on algebraic varieties:
\begin{theorem}\label{theo:rep.KKT}
Let $h_0,h_1,\dots,h_l$ be polynomial in $\R_d[x]$  with $d\ge 2$.
Assume that $h_0$ is non-negative on $V(h)$ with $h=(h_1,\dots,h_l)$.
Set $\bar h:=(h_0,h)$, $\lambda:=(\lambda_1,\dots,\lambda_l)$ and
\begin{equation}\label{eq:def.w.KKT}
w:=\max\{d\times (c(n+l,d,n+l)-1),\frac{1}{2}\times b(n+l,d,l+n+1)+d\}\,,
\end{equation}
where $c(\cdot)$ and $b(\cdot)$ are defined as in \eqref{eq:bound.connected} and  \eqref{eq:bound.Krivine}, respectively.
Then the following statements hold:
\begin{enumerate}
\item The cardinality of $h_0(V(\bar h_{\KKT}))$ is at most $c(n+l,d+1,n+l)$.
\item There exists $\sigma\in \Sigma^2_w[x, \lambda]$ such that $h_0-\sigma$ vanishes on $V(\bar h_{\KKT})$.
\end{enumerate}
\end{theorem} 
\begin{proof}
Using Lemma \ref{lem:num.connected}, we decompose $V(\bar h_{\KKT})$ into semi-algebraically path connected components:
$Z_1,\dots,Z_s$ with  
\begin{equation}\label{eq:bound.on.s.KKT}
s\le c(n+l,d,n+l)\,,
\end{equation}
since each entry of $\bar h_{\KKT}$ has degree at most $d\ge 2$.
Accordingly Lemma \ref{lem:constant.KKT} shows that $h_0$ is constant on each $Z_i$.
Thus $h_0(V(\bar h_{\KKT}))$ is finite.
Set $r=|h_0(V(\bar h_{\KKT}))|$.
From \eqref{eq:bound.on.s.KKT}, we get 
\begin{equation}\label{eq:ineq.KKT}
r\le s\le c(n+l,d,n+l)\,.
\end{equation}
Set 
\begin{equation}
u:=\frac{1}{2}\times b(n+l,d,l+n+1)\,.
\end{equation}
By using Lemma \ref{lem:quadra}, there exists $\sigma\in \Sigma^2_\xi[x,\lambda]$ with $\xi=\max\{d(r-1),u+d\}$ such that $h_0 - \sigma$ vanishes on $V(\bar h_{\KKT})$.
By \eqref{eq:ineq.KKT} and \eqref{eq:def.w.KKT}, we get $\xi\le w$, and hence $\sigma\in \Sigma^2_w[x,\lambda]$. 
\end{proof}

\subsection{Exact semidefinite programs in the regular case}

We recall some preliminaries  of the Moment-SOS relaxations originally developed by Lasserre in \cite{lasserre2001global}.
Given $d\in\N$, let $\N^n_d:=\{\alpha\in\N^n\,:\,\sum_{j=1}^n \alpha_j\le d\}$.
Given $d\in\N$, we denote by $v_d$ the vector of monomials in $x$ of degree at most $d$, i.e., $v_d=(x^\alpha)_{\alpha\in\N^n_d}$ with $x^\alpha:=x_1^{\alpha_1}\dots x_n^{\alpha_n}$.
For each $p\in\R_d[x]$, we write $p=c(p)^\top v_d=\sum_{\alpha\in\N^n_d}p_\alpha x^\alpha$, where $c(p)$ is denoted by the vector of coefficient of $p$, i.e., $c(p)=(p_\alpha)_{\alpha\in\N^n_d}$ with $p_\alpha\in\R$.
Given $A\in\R^{r\times r}$ being symmetric, we say that $A$ is positive semidefinite, denoted by $A\succeq 0$, if every eigenvalue of $A$ is non-negative.

Given $y=(y_\alpha)_{\alpha\in\N^n}\subset \R$, let $L_y:\R[x]\to\R$ be the Riesz linear functional defined by $L_y(p)=\sum_{\alpha\in\N^n} p_\alpha y_\alpha$ for every $p\in\R[x]$.
Given $d\in\N$, $p\in\R[x]$ and $y=(y_\alpha)_{\alpha\in\N^n}\subset \R$, let $M_d(y)$ be the moment matrix of order $d$ defined by $(y_{\alpha+\beta})_{\alpha,\beta\in\N^n_d}$.

The following lemma shows the connection between sums of squares and semidefinite programming (see, e.g., \cite[Proposition 2.1]{lasserre2015introduction}):
\begin{lemma}\label{lem:sdp.sos}
Let $\sigma\in\R[x]$ and $d\in\N$ such that $2d\ge\deg(\sigma)$.
Then $\sigma\in\Sigma^2[x]$ iff there exists $G\succeq 0$ such that $\sigma=v_d^\top Gv_d$.
\end{lemma}
Given $k\in\N$ and $h_0,h_1,\dots,h_l\in\R[x]$, consider the following primal-dual semidefinite programs associated with $\bar h=(h_0,h_1,\dots,h_l)$:
\begin{equation}\label{eq:mom.relax}
\begin{array}{rl}
\tau_k(\bar h):=\inf\limits_y& L_y(h_0)\\
\text{s.t} &M_k(y)\succeq 0\,,\\
&M_{k-r_t}(h_ty)=0\,,\,t=1,\dots,l\,,\,y_0=1\,,
\end{array}
\end{equation}
\begin{equation}\label{eq:sos.relax}
\begin{array}{rl}
\rho_k(\bar h):=\sup\limits_{\xi,G,u_t} & \xi\\
\text{s.t} & G\succeq 0\,,\\
&h_0-\xi=v_k^\top Gv_k+\sum_{t=1}^l h_tu_t^\top v_{2(k-r_t)}\,,\\
\end{array}
\end{equation}
where $r_t=\lceil \deg(h_t)/2\rceil$.
Using Lemma \ref{lem:sdp.sos}, we obtain 
\begin{equation}\label{eq:equi.sos}
\rho_k(\bar h)=\sup_{\xi\in\R}\{ \xi\,:\,h_0-\xi\in \Sigma_k[x]+I_k(h)[x]\}\,.
\end{equation}
Primal-dual semidefinite programs \eqref{eq:mom.relax}-\eqref{eq:sos.relax} is known as the Moment-SOS relaxations of order $k$ for problem \eqref{eq:pop}.

We state in the following lemma some recent results involving the Moment-SOS relaxations:
\begin{lemma}\label{lem:mom.sos}
Let $h_0,h_1,\dots,h_l\in\R_{d}[x]$. 
Let $h^\star$ be as in \eqref{eq:pop}  with $h=(h_1,\dots,h_l)$. 
Set $\bar h=(h_0,h)$.
Assume that $h^\star>-\infty$.
Then the following statements hold:
\begin{enumerate}
\item For every $k\in\N$, $\tau_k(\bar h)\le \tau_{k+1}(\bar h)$ and $\rho_k(\bar h)\le \rho_{k+1}(\bar h)$.
\item For every $k\in\N$, $\rho_k(\bar h)\le \tau_{k}(\bar h)\le h^\star$.
\item If there exists $q\in \Sigma_w[x]$ with $2w\ge d$ such that $h_0-h^\star-q$ vanishes on $V(h)$, then  $\rho_r(\bar h)=h^\star$ with $r=\frac{1}{2}\times b(n,2w,l+1)+d$, where  $b(\cdot)$ is defined as in   \eqref{eq:bound.Krivine}.
\end{enumerate}
\end{lemma}
\begin{proof}
The proofs of the first two statements are trivial.
Let us use Nie's technique in \cite[Proof of Theorem 1.1]{nie2014optimality} to prove the third statement. 
Consider the following two cases:
\begin{itemize}
\item Case 1: $V(h)= \emptyset$. 
Then we get $h^\star=\infty$.
Set $s=\frac{1}{2}\times b(n,d,l+1)$.
Then Lemma \ref{lem:pos2} says that $-1 \in \Sigma^2_s[x]+I_s(h)[x]$.
For all $\xi\ge 0$, it holds that
\begin{equation}
h_0-\xi=(1+\frac{h_0}{4})^2-(\xi+(1-\frac{h_0}{4})^2)\in \Sigma_{s+d}[x]+I_{s+d}(h)[x]\,.
\end{equation}
Since $r\ge s+d$, it implies that for all $\xi\ge 0$, $h_0-\xi\in \Sigma_{r}[x]+I_{r}(h)[x]$, which yields that $\xi$ is a feasible solution for \eqref{eq:equi.sos} of the value $\rho_r(\bar h)$.
Thus we obtain that $\rho_r(\bar h)=\infty= h^\star$.
\item Case 2: $V(h)\ne \emptyset$. 
Then we get $h^\star<\infty$.
Set $u=h_0-h^\star-q$.
By assumption,we get $u\in\R_{2w}[x]$ and $u=0$ on $V(h)$. 
Set $s=2\lfloor (r-d)/(2w)\rfloor$.
From this, Lemma \ref{lem:pos2} says that there exist  $\sigma \in \Sigma_{r-d}^2[x]$ such that $u^{2s} + \sigma \in I_{r-d}(h)[x]$.
Let $c=\frac{1}{2s}$. 
Then it holds that $1+t+ct^{2s}\in\Sigma^2_s[t]$.
Thus for all $\varepsilon>0$, we have
\begin{equation}
\begin{array}{rl}
h_0-h^\star+\varepsilon&=q + \varepsilon(1+\frac{u}{\varepsilon}+c\left(\frac{u}{\varepsilon}\right)^{2s})-c\varepsilon^{1-2s}(u^{2s} + \sigma ) +c\varepsilon^{1-2s}\sigma\\
&\in \Sigma^2_{r-d}[x]+I_{r-d}(h)[x]\subset \Sigma^2_{r}[x]+I_{r}(h)[x]\,.
\end{array}
\end{equation}
Then we for all $\varepsilon>0$, $h^\star-\varepsilon$ is a feasible solution of \eqref{eq:equi.sos} of the value $\rho_r(\bar h)$.
It gives $\rho_r(\bar h)\ge h^\star-\varepsilon$, for all $\varepsilon>0$, and, in consequence, we get $\rho_r(\bar h)\ge h^\star$.
Using the second statement, we obtain that $\rho_r(\bar h)= h^\star$, yielding the third statement.
\end{itemize}
\end{proof}

We apply Theorem \ref{theo:rep.KKT} for polynomial optimization as follows:
\begin{theorem}\label{theo:pop.KKT}
Let $h_0,h_1,\dots,h_l\in\R_d[x]$. 
Let $h^\star$ be as in problem \eqref{eq:pop} with $h=(h_1,\dots,h_l)$.
Assume that problem \eqref{eq:pop} has a global minimizer at which the Karush--Kuhn--Tucker conditions hold for this problem.
Let $w$ be as in \eqref{eq:def.w.KKT}.
Set 
\begin{equation}\label{eq:r.formula}
r:=\frac{1}{2}\times b(n+l,2w,l+n+1)+d\,.
\end{equation}
Then $\rho_r(h_0,\bar h_{\KKT})=h^\star$, where $\bar h:=(h_0,h)$ and $b(\cdot)$ is defined as in  \eqref{eq:bound.Krivine}.
\end{theorem}
\begin{proof}
By assumption, there exists $(x^\star,\lambda^\star)\in V(\bar h_{\KKT})$ such that $x^\star$ is a global minimizer of \eqref{eq:pop}.
It implies that
\begin{equation}
\begin{array}{rl}
h^\star:=\min\limits_{x,\lambda}& h_0(x)\\
\text{s.t.}& (x,\lambda)\in V(\bar h_{\KKT})\,,
\end{array}
\end{equation}
By assumption, Theorem \ref{theo:rep.KKT} yields that there exists $q\in \Sigma^2_w(g)[x,\lambda]$ such that $h_0-h^\star-q$ vanishes on $V(\bar h_{\KKT})$.
Applying the third statement of Lemma \ref{lem:mom.sos}, we obtain the conclusion.
\end{proof}
\begin{remark}\label{rem:not.attained}
Let $h=(h_1,\dots,h_l)$ with $h_j\in\R[x]$ and let $h^\star$ be as in \eqref{eq:pop}.
Assume that $h^\star$ is finite but is not attained i.e., $h_0(x)-h^\star>0$ for all $x\in V(h)$. (For instance, we can take (i) $h_0=x_1$ and $h=(x_1x_2^2-1)$ or (ii) $h_0=(x_1x_2-1)^2 + x_1^2$ and $h=(0)$.)
By Theorem \ref{theo:rep.KKT}, the set $h_0(V(\bar h_{\KKT}))-h^\star$ (with $\bar h:=(h_0,h)$) has a finite number of values but does not have zero value.
It is because of 
\begin{equation}
0\notin h_0(V(h))-h^\star\supset h_0(V(\bar h_{\KKT}))-h^\star\,.
\end{equation}
It implies that $\inf (h_0(V(\bar h_{\KKT}))-h^\star)=\delta>0$, so that $\inf h_0(V(\bar h_{\KKT}))=h^\star+\delta$.
Note that $\delta=\infty$ iff $h_0(V(\bar h_{\KKT}))=\emptyset$.
Thus we obtain $\rho_r(h_0,\bar h_{\KKT})=h^\star+\delta>h^\star$, where $w$ is as in \eqref{eq:def.w.KKT} and $r$ is as in \eqref{eq:r.formula}.

To address this attainability issue, the author proves in \cite{mai2022semi} that every polynomial optimization problem of the form \eqref{eq:pop} with finite infimum value $h^\star$ can be symbolically transformed to an equivalent problem in one-dimensional space with attained optimal value $h^\star$.
To do this, he uses quantifier elimination and algebraic algorithms that rely on the fundamental theorem of algebra and the greatest common divisor.
Let $d$ be the upper bound on the degrees of $h_j$. 
His symbolic algorithm has complexity $O(d^{O(n)})$ to produce the objective and constraint polynomials of degree at most $d^{O(n)}$ for the equivalent problem.
\end{remark}

\section{Nichtnegativstellensatz on singular varieties}
\label{sec:represent.sing}

We provide in the following theorem the sums of squares-based representations on varieties whose singular loci are dense:

\begin{theorem}\label{theo:rep.nonneg}
Let $V$ be a real algebraic variety in $\R^n$ with regular locus $V^{\reg}$ and $f$ be a polynomial in $\R[x]$.
Assume that $V^{\reg}$ is dense in $V$.
Then the following statements hold:
\begin{enumerate}
\item There exists a regular variety $W$ in $\R^{t}$ together with a proper birational morphism $\varphi:W\to V$ such that the restriction $U\to V^{\reg}\,,\,y\mapsto \varphi(y)$, is an isomorphism for some $U\subset W$.
Moreover, $W$ and $\varphi$ are algorithmically obtained.
\item Set $h_0=f\circ \varphi $.
Suppose that $h_0$ attains its infimum value on $W$.
Let $h_1,\dots,h_l\in\R[y]$ be the generators of the vanishing ideal $I(W)$.
Let $d\in\N$ with $d\ge 2$ be the upper bound on the degrees of $h_j$s.
Set $h=(h_1,\dots,h_l)$, $\bar h:=(h_0,h)$, $\lambda:=(\lambda_1,\dots,\lambda_l)$ and
\begin{equation}\label{eq:def.u.KKT}
u:=\max\{d\times (c(t+l,d,t+l)-1),\frac{1}{2}\times b(t+l,d,l+t+1)+d\}\,,
\end{equation}
where $c(\cdot)$ and $b(\cdot)$ are defined as in \eqref{eq:bound.connected} and  \eqref{eq:bound.Krivine}, respectively.
Then the following conditions are equivalent:
\begin{enumerate}
\item The polynomial $f$ is non-negative on $V$.
\item There exists $\sigma\in \Sigma^2_u[x, \lambda]$ such that $h_0-\sigma$ vanishes on $V(\bar h_{\KKT})$.
\end{enumerate}
\end{enumerate}
\end{theorem}
\begin{proof}
The first statement is given in Lemma \ref{lem:resol.sing} and Remark \ref{rem:alg.resol}.
Let us prove the second one.
We claim that (a) is equivalent to condition (c) saying that $f\circ \varphi$ is non-negative on $W$.
Since $\varphi(W)\subset V$, (a) implies (c).
Assume that (c) holds.
We claim (d) saying that $f$ is non-negative on $V^{\reg}$.
Indeed, take $x\in V^{\reg}$. 
Since the restriction $U\to V^{\reg}\,,\,y\mapsto \varphi(y)$, is an isomorphism, there exists $y\in U\subset W$ such that $x=\varphi(y)$, which gives $f(x)=(f\circ \varphi)(y)\ge 0$ by (c).
Let $z\in V$.
Since $V^{\reg}$ is dense in $V$, there is a sequence $(z^{(j)})_{j=1}^\infty\subset V^{\reg}$ which converges to $z$.
By (d), $f(z^{(j}))\ge 0$ for $j=1,2,\dots$.
By the continuity of $f$, we get $f(z)\ge 0$, which implies (a).

We now prove that (c) is equivalent to (b).
By Theorem \ref{theo:rep.KKT}, (c) implies (b).
Assume that (b) holds.
Note that $W=V(h)$.
Let $h^\star$ be as in \eqref{eq:pop} and $y^\star$ is a global minimizer of \eqref{eq:pop}.
Since $W$ is regular, the Jacobian matrix $J(h)(y^\star)$ has rank $n-\dim(W)$.
From this, Lemma \ref{lem:KKT} yields that the Karush--Kuhn--Tucker conditions hold for problem \eqref{eq:pop} at $y^\star$.
It means that there exists $(y^\star,\lambda^\star)\in V(\bar h_{\KKT})$.
From this and (b), we obtain $h^\star=h_0(y^\star)=\sigma(y^\star,\lambda^\star)\ge 0$, which implies (c).

Since (a) is equivalent to (c) and (c) is equivalent to (b), (a) is equivalent to (b), yielding the result. 
\end{proof}
\begin{remark}
 In Theorem \ref{theo:rep.nonneg}, the denseness assumption of $V^{\reg}$ in $V$ does not always hold even if $V$ is irreducible, as shown in Examples \ref{exam:not.dense} and \ref{exam:not.dense2}.
(However, if $\bar V$ is an irreducible complex algebraic variety, the regular locus of $\bar V$ is dense in $\bar V$ (see \cite[Chapter 7, Section 2.1, Lemma 1]{igor1974basic}).)
Examples \ref{exam:cuspidal.plane}, \ref{exam:lorentz.cone}, and \ref{exam:infinite.sing} indicate cases where this assumption holds even if the singular locus $V^{\sing}$ of $V$ is infinite.
\end{remark} 
\begin{remark}
In Theorem \ref{theo:rep.nonneg}, it still holds when we remove the denseness assumption of $V^{\reg}$ in $V$ and modify condition (a) as ``The polynomial $f$ is non-negative on $V^{\reg}$.''.
However, we cannot characterize the non-negativity of $f$ on $V^{\sing}$ in this case.
Here $V^{\sing}$ stands for the singular locus of $V$.
\end{remark}
\begin{remark}
In Theorem \ref{theo:rep.nonneg}, it still holds when we remove the denseness assumption of $V^{\reg}$ in $V$ and add the assumption saying that $\varphi$ is surjective i.e., $V=\varphi(W)$.
In this case, we get $f(V)=(f\circ \varphi)(W)$, which implies the equivalence of the non-negativities of $f$ on $V$ and $f\circ \varphi$ on $W$.
Moreover, this additional assumption can be replaced with ``$\varphi$ is dominant'', i.e., $V\subset \overline{\varphi(W)}$.
Here $\overline{A}$ represents for the closure of a subset $A\subset \R^n$ (by the usual topology).
\end{remark}
\begin{remark}
In Theorem \ref{theo:rep.nonneg}, it still holds when we remove the denseness assumption of $V^{\reg}$ in $V$ and replace the first statement with ``There exists a regular variety $W$ in $\R^{t}$ together with a proper dominant morphism $\varphi:W\to V$.'' (see Remark \ref{rem:de.Jong}).
Note that no algorithm allows us to get such $W$ and $\varphi$.
\end{remark}
We illustrate the result of Theorem \ref{theo:rep.nonneg} in the following simple example:
\begin{example}\label{exam:rep}
Let $f=x_1$, and $p,h,V,W,V^{\reg},V^{\sing},\varphi$ be as in Example \ref{exam:resol.sing.1}.
Then $f\ge 0$ on $V$.
The unique real zeros of $f$ on $V$ is the origin, which is the unique singular point of $V$.
By Example \ref{exam:KKT.not.hold}, the Karush--Kuhn--Tucker conditions do not hold for problem $\min\limits_{x\in V(p)} f(x)$ at this point.
Note that $V=V(p)$.
Set $h_0=f\circ \varphi$. Then $h_0(y)=y_1$.
It is clear that $h_0$ is non-negative on $W$ and has unique real zero $(0,0)$ on $W$.
Moreover, the Karush--Kuhn--Tucker conditions hold for problem $\min\limits_{y\in V(h)} h_0(y)$ at this point.
Note that $W=V(h)$.
Set $\bar h=(h_0,h)$.
Theorem \ref{theo:rep.nonneg} yields that there exists $\sigma\in \Sigma^2_u[y,\lambda]$ with $u\in\N$ such that $h_0-\sigma$ vanishes on $V(\bar h_{\KKT})$ with $\bar h_{\KKT}$ being given by
\begin{equation}
\bar h_{KKT}=(h,\nabla h_0-\lambda\nabla h)=(y_1-y_2^2,\begin{bmatrix}
1\\
0
\end{bmatrix}-\lambda\begin{bmatrix}
1\\
-2y_2
\end{bmatrix})\,.
\end{equation}
Let $(y,\lambda)\in V(\bar h_\text{KKT})$.
Then we get $\lambda=1$ and $y_2=0$, which gives $y_1=0$. 
Thus give $V(\bar h_{\KKT})=\{(0,0,1)\}$.
It is not hart to check that $\sigma=0$ satisfies $\sigma\in \Sigma^2_u[y,\lambda]$ and $h_0-\sigma$ vanishes on $V(\bar h_{\KKT})$.
\end{example}

\section{Exact semidefinite programs in the singular case}
\label{sec:POP.general}

We present the application of the resolution of singularities to polynomial optimization in the following theorem:

\begin{theorem}\label{theo:pop}
Let $V$ be a real algebraic variety in $\R^n$ with regular locus $V^{\reg}$ and $q$ be a polynomial in $\R[x]$.
Assume that $V^{\reg}$ is dense in $V$. 
Consider polynomial optimization problem:
\begin{equation}\label{eq:ori.pop}
h^\star:=\inf_{x\in V} q(x)\,.
\end{equation}
Then the following statements hold:
\begin{enumerate}
\item There exists a regular variety $W$ in $\R^{t}$ together with a proper birational morphism $\varphi:W\to V$ such that the restriction $U\to V^{\reg}\,,\,y\mapsto \varphi(y)$, is an isomorphism for some $U\subset W$.
Moreover, $W$ and $\varphi$ are algorithmically obtained.
\item Set $h_0=q\circ \varphi $.
Suppose that $h_0$ attains its infimum value on $W$.
Let $h_1,\dots,h_l\in\R[y]$ be the generators of the vanishing ideal $I(W)$.
Let $d\in\N$ with $d\ge 2$ be the upper bound on the degrees of $h_j$s.
Set $h=(h_1,\dots,h_l)$, $\bar h:=(h_0,h)$, $\lambda:=(\lambda_1,\dots,\lambda_l)$ and $u$ be as in \eqref{eq:def.u.KKT}.
Set 
\begin{equation}\label{eq:r.formula.app}
r:=\frac{1}{2}\times b(t+l,2u,l+t+1)+d\,.
\end{equation}
Then $\rho_r(h_0,\bar h_{\KKT})=h^\star$.
\end{enumerate}
\end{theorem}
\begin{proof}
The first statement is given in Lemma \ref{lem:resol.sing}.
Let us prove the second one.
We show that 
\begin{equation}\label{eq:equi.pop}
h^\star:=\inf_{y\in W} h_0(x)\,.
\end{equation}
By \eqref{eq:ori.pop}, $q-h^\star\ge 0$ on $V$.
The same procedure as in the proof of Theorem \ref{theo:rep.nonneg} gives $h_0-h^\star\ge 0$ on $W$.
It remains to find a sequence $(y^{(j)})_{j=1}^\infty\subset W$ such that $h_0(y^{(j)})\to h^\star$ as $j\to\infty$.
Let $(x^{(j)})_{j=1}^\infty$ be a minimizing sequence of problem \eqref{eq:ori.pop}, i.e., $x^{(j)}\in V$ and $q(x^{(j)})\to h^\star$ as $j\to \infty$.
By the continuity of $q$ and the denseness of $V^{\reg}$ in $V$, let $z^{(j)}\in V^{\reg}$ in a sufficiently small neighborhood of $x^{(j)}$ such that $|q(z^{(j)})-q(x^{(j)})|\le \frac{1}{j}$.
We claim that $(z^{(j)})_{j=1}^\infty$ is also a minimizing sequence of problem \eqref{eq:ori.pop}.
Indeed, we get $z^{(j)}\in V$ and
\begin{equation}
|q(z^{(j)})-h^\star|\le |q(z^{(j)})-q(x^{(j)})|+|q(x^{(j)})-h^\star|\le \frac{1}{j}+|q(x^{(j)})-h^\star|\to 0
\end{equation}
as $j\to \infty$.
Since the restriction $U\to V^{\reg}\,,\,y\mapsto \varphi(y)$, is an isomorphism, there is $y^{(j)}\in U\subset W$ such that $z^{(j)}=\varphi(y^{(j)})$.
It implies that $h_0(y^{(j)})=q(z^{(j)})\to h^\star$ as $j\to\infty$.
By assumption, a global minimizer $y^\star$ of \eqref{eq:pop} exists.
Note that $W=V(h)$.
Since $W$ is regular, the Jacobian matrix $J(h)(y^\star)$ has rank $n-\dim(W)$.
From this, Lemma \ref{lem:KKT} yields that the Karush--Kuhn--Tucker conditions hold for problem \eqref{eq:pop} at $y^\star$.
By Theorem \ref{theo:pop.KKT}, the result follows.
\end{proof}





\paragraph{Acknowledgements.} 
The author was supported by the funding from ANITI.

\bibliographystyle{abbrv}

\end{document}